\documentclass[12pt,reqno]{article}

\usepackage[usenames]{color}
\usepackage{amssymb}
\usepackage{amsmath}
\usepackage{amsthm}
\usepackage{amsfonts}
\usepackage{amscd}
\usepackage{graphicx}
\usepackage{enumitem}

\usepackage[colorlinks=true,
linkcolor=webgreen,
filecolor=webbrown,
citecolor=webgreen]{hyperref}

\definecolor{webgreen}{rgb}{0,.5,0}
\definecolor{webbrown}{rgb}{.6,0,0}

\usepackage{color}
\usepackage{fullpage}
\usepackage{float}

\usepackage{graphics}
\usepackage{latexsym}
\usepackage{epsf}
\usepackage{breakurl}

\begin{document}


\theoremstyle{plain}
\newtheorem{theorem}{Theorem}
\newtheorem{corollary}[theorem]{Corollary}
\newtheorem{lemma}[theorem]{Lemma}
\newtheorem{proposition}[theorem]{Proposition}

\theoremstyle{definition}
\newtheorem{definition}[theorem]{Definition}
\newtheorem{example}[theorem]{Example}
\newtheorem{conjecture}[theorem]{Conjecture}

\theoremstyle{remark}
\newtheorem{remark}[theorem]{Remark}

\begin{center}
\vskip 1cm{\Large\bf
Sums of Powers of Integers and the Sequence \\
\vskip .1in
A304330}
\vskip 1cm
\large
Jos\'{e} L. Cereceda\\
Collado Villalba, 28400 (Madrid)\\
Spain\\
\href{mailto:jl.cereceda@movistar.es}{\tt jl.cereceda@movistar.es} \\
\end{center}

\vskip .2in

\begin{abstract}
For integer $k \geq 1$, let $S_k(n)$ denote the sum of the $k$th powers of the first $n$ positive integers. In this paper, we derive a new formula expressing $2^{2k}$ times $S_{2k}(n)$ as a sum of $k$ terms involving the numbers in the $k$th row of the integer sequence A304330, which is closely related to the central factorial numbers with even indices of the second kind. Furthermore, we provide an alternative proof of Knuth's formula for $S_{2k}(n)$ and show that it can equally be expressed in terms of A304330. Moreover, we obtain corresponding formulas for $2^{2k-1}S_{2k-1}(n)$ and determine the Faulhaber form of both $S_{2k}(n)$ and $S_{2k+1}(n)$ in terms of A304330 and the Legendre-Stirling numbers of the first kind.
\end{abstract}

\section{Introduction and main results}

For integer $k \geq 1$, consider the sum of the $k$th powers of the first $n$ positive integers $S_k(n) = 1^k + 2^k + \cdots + n^k$. Recently, the author \cite[Theorem 2]{cere} has shown that $S_{2k}(n)$ can be expressed in the form
\begin{equation}\label{int1}
S_{2k}(n) = \frac{1}{2} \sum \frac{(2k)!}{b_1! b_2! \cdots b_k!} \prod_{r=1}^{k} \biggl( \frac{1}{4^r (2r)!}
\biggl )^{b_r} 2^m m! \binom{2n+m+1}{2m+1},
\end{equation}
where the summation takes place over all $k$-tuples of nonnegative integers $(b_1,b_2,\ldots,b_k)$ satisfying the constraint $b_1 + 2b_2 + \cdots + k b_k = k$, and where $m = b_1 + b_2 +\cdots +b_k$. Since $m$ runs over all the integers in the set $\{1,2,\ldots,k \}$, the above formula implies that $S_{2k}(n)$ can be expressed in the polynomial form
\begin{equation*}
S_{2k}(n) = \sum_{m=1}^k p_{k,m} \binom{2n+m+1}{2m+1},
\end{equation*}
for certain (nonzero) rational coefficients $p_{k,1}, p_{k,2}, \ldots, p_{k,k}$ which are independent of $n$. Indeed, as can be checked for $k =1,2,3,4$, formula \eqref{int1} yields
\begin{align*}
S_2(n) & = \frac{1}{4} \binom{2n+2}{3}, \\
S_4(n) & = \frac{1}{16} \binom{2n+2}{3} + \frac{3}{4} \binom{2n+3}{5}, \\
S_6(n) & = \frac{1}{64} \binom{2n+2}{3} + \frac{15}{16} \binom{2n+3}{5} + \frac{45}{8} \binom{2n+4}{7}, \\
S_8(n) & = \frac{1}{256} \binom{2n+2}{3} + \frac{63}{64} \binom{2n+3}{5} + \frac{315}{16} \binom{2n+4}{7}
+ \frac{315}{4} \binom{2n+5}{9}.
\end{align*}
If we now multiply these equations by the corresponding factor $2^{2k}$, $k =1,2,3,4$, we are left with the following all-integer formulas:
\begin{align*}
2^2 S_2(n) & = \binom{2n+2}{3}, \\
2^4 S_4(n) & = \binom{2n+2}{3} + 12 \binom{2n+3}{5}, \\
2^6 S_6(n) & = \binom{2n+2}{3} + 60 \binom{2n+3}{5} + 360 \binom{2n+4}{7}, \\
2^8 S_8(n) & = \binom{2n+2}{3} + 252 \binom{2n+3}{5} + 5040 \binom{2n+4}{7} + 20160 \binom{2n+5}{9}.
\end{align*}

At this point, it is pertinent to bring up the sequence A304330 in the On-Line Encyclopedia of Integer Sequences (OEIS) \cite{OEIS}, whose general term is given by
\begin{equation}\label{exprkm}
R(k,m) = \sum_{j=0}^m (-1)^j \binom{2m}{j} (m-j)^{2k}, \quad 0 \leq m \leq k.
\end{equation}
Table \ref{tab:1} displays the first few rows of the numerical triangle for the sequence A304330. As we will see now, there is an intimate connection between this sequence and the sequence of central factorial numbers with even indices of the second kind $T(2k,2m)$, listed as A036969 in the OEIS. Following the notation introduced by Gelineau and Zeng \cite{gelineau}, in this paper we will write $U(k,m)$ to refer to $T(2k,2m)$. As shown by Butzer et al.\ \cite[Proposition 2.4 (xiii)]{butzer}, the numbers $U(k,m)$ admit the explicit formula
\begin{equation*}
U(k,m) = \frac{2}{(2m)!} \sum_{j=0}^m (-1)^{m+j} \binom{2m}{m-j} j^{2k},
\end{equation*}
or, equivalently,
\begin{equation*}
U(k,m) = \frac{2}{(2m)!} \sum_{j=0}^m (-1)^j \binom{2m}{j} (m-j)^{2k},
\end{equation*}
where it is assumed that $0 \leq m \leq k$. Thus, $R(k,m)$ and $U(k,m)$ are related by
\begin{equation}\label{int2}
R(k,m) = \frac{(2m)!}{2} U(k,m).
\end{equation}

\begin{table}[ttt]
\centering
\begin{tabular}{|c||c|c|c|c|c|c|c|}
\hline
$k \backslash m$ & $m=0$ & $m=1$ & $m=2$ & $m=3$ & $m=4$ & $m=5$ & $m=6$ \\
\hline\hline
$k=0$ & 1 & 0 & 0 & 0 & 0 & 0 & 0  \\ \hline
$k=1$ & 0 & 1 & 0 & 0 & 0 & 0 & 0  \\ \hline
$k=2$ & 0 & 1 & 12 & 0 & 0 & 0 & 0  \\ \hline
$k=3$ & 0 & 1 & 60 & 360 & 0 & 0 & 0  \\ \hline
$k=4$ & 0 & 1 & 252 & 5040 & 20160 & 0 & 0  \\ \hline
$k=5$ & 0 & 1 & 1020 & 52920 & 604800 & 1814400 & 0  \\ \hline
$k=6$ & 0 & 1 & 4092 & 506880 & 12640320 & 99792000 & 239500800  \\ \hline
\end{tabular}
\caption{Triangular array of the numbers $R(k,m)$ up to $k=6$.}\label{tab:1}
\end{table}

We can see that the integer coefficients that appear in the previous expressions of $2^{2k} S_{2k}(n)$, $k =1,2,3,4$, are precisely the entries in the $k$th row of the triangle in Table~\ref{tab:1}. We could then conjecture that this rule effectively applies for successive values of $k$ and, eventually, for every $k$. The following theorem (which is proven in Section \ref{sec:2}) confirms that this is indeed the case.
\begin{theorem}\label{th:1}
For integer $k \geq 1$, we have
\begin{equation}\label{th1}
2^{2k} S_{2k}(n) = \sum_{m=1}^k R(k,m) \binom{2n+m+1}{2m+1}.
\end{equation}
\end{theorem}

On the other hand, in his excellent paper on Johann Faulhaber and power sums, Knuth \cite[p.\ 285]{knuth} derived the following alternative, explicit formulas for $S_{2k}(n)$, $k=1,\ldots,6$:
\begin{align*}
S_2(n) & = T_1(n),  \\
S_4(n) & = T_1(n) + 12 T_2(n), \\
S_6(n) & = T_1(n) + 60 T_2(n) + 360 T_3(n), \\
S_8(n) & = T_1(n) + 252 T_2(n) + 5040 T_3(n) + 20160 T_4(n), \\
S_{10}(n) & =  T_1(n) + 1020 T_2(n) + 52920 T_3(n) + 604800 T_4(n) + 1814400 T_5(n), \\
S_{12}(n) & =  T_1(n) + 4092 T_2(n) + 506880 T_3(n) + 12640320 T_4(n) + 99792000 T_5(n) \\
& \qquad \qquad \qquad \qquad \quad \,\, + 239500800 T_6(n),
\end{align*}
where
\begin{equation*}
T_m(n) = \frac{2n+1}{2m+1} \binom{n+m}{2m}.
\end{equation*}

Knuth also found \cite[pp.\ 285--286]{knuth} the following relationship between the power sums $S_{2k}(n)$ and $S_{2k-1}(n)$.
\begin{proposition}[Knuth, 1993]\label{pr:1}
The formula
\begin{equation*}
\frac{S_{2k}(n)}{2n+1} = a_1 \binom{n+1}{2} + a_2 \binom{n+2}{4} + \cdots + a_k \binom{n+k}{2k},
\end{equation*}
holds if and only if
\begin{equation*}
S_{2k-1}(n) = \frac{3}{1}a_1 \binom{n+1}{2} + \frac{5}{2}a_2 \binom{n+2}{4} + \cdots +
\frac{2k+1}{k} a_k \binom{n+k}{2k}.
\end{equation*}
\end{proposition}

Since the coefficients $a_1, a_2, \ldots, a_k$ for $S_{2k-1}(n)$ are known (see the formula for $S_{2k-1}(n)$ at the bottom of p.\ 284 of Knuth's paper \cite{knuth}), Proposition \ref{pr:1} enables one to fully determine $S_{2k}(n)$ for arbitrary $k$. The explicit expressions of the above formulas for $S_{2k}(n)$ and $S_{2k-1}(n)$ are then given as follows.
\begin{proposition}[Knuth, 1993]\label{pr:2}
For integer $k\geq 1$, we have
\begin{align}
S_{2k}(n) & = \sum_{m=1}^k R(k,m) \, \frac{2n+1}{2m+1} \binom{n+m}{2m}, \label{pr21} \\[-2mm]
\intertext{and}
S_{2k-1}(n) & = \sum_{m=1}^k  R(k,m) \, \frac{1}{m} \binom{n+m}{2m}. \label{pr22}
\end{align}
\end{proposition}

The rest of the paper is organized as follows. In Section \ref{sec:2}, we firstly prove the above formulas \eqref{th1}, \eqref{pr21}, and \eqref{pr22}. As in the case of \eqref{th1}, our proof of \eqref{pr21} and \eqref{pr22} relies on the respective binomial identity as well as on a pair of properties of the central factorial numbers with even indices. As a by-product, we obtain formulas for the power sums $T_{2k}(n) = \sum_{i=1}^n (2i -1)^{2k}$ and $\Omega_{2k} = \sum_{i=1}^n (-1)^{n-i} i^{2k}$. Furthermore, we derive a couple of formulas for the product $2^{2k-1} S_{2k-1}(n)$ involving the numbers $R(k,m)$ (Theorem \ref{th:9} and equation \eqref{odd2}, respectively). In Section~\ref{sec:3}, we deal with the so-called Faulhaber form of the power sums $S_{2k}(n)$ and $S_{2k+1}(n)$ and provide an alternative representation of the Faulhaber coefficients in terms on $R(k,m)$ and the Legendre-Stirling numbers of the first kind (Proposition \ref{prop:11}). We conclude the paper in Section \ref{sec:4} with some additional remarks.

Next we summarize the main results found in this paper, highlighting the close relationship between the integer sequence A304330 (with general term given by \eqref{exprkm}) and the three varieties of power sums considered in this paper, namely, $S_k(n) = 1^k + 2^k + \cdots + n^k$, $T_k(n) = 1^k + 3^k + \cdots + (2n-1)^k$, and $\Omega_k(n) = n^k - (n-1)^k + \cdots + (-1)^{n-1} 1^k$. (Note that the formulas for $S_{2k}(n)$ and $S_{2k-1}(n)$ were previously obtained in Knuth's paper \cite{knuth}.)

For integer $k \geq 1$, we have
\begin{align*}
2^{2k} S_{2k}(n) & = \sum_{m=1}^k R(k,m) \binom{2n+m+1}{2m+1},  \\
S_{2k}(n) & = \sum_{m=1}^k R(k,m) \, \frac{2n+1}{2m+1} \binom{n+m}{2m}, \\
T_{2k}(n) & = \sum_{m=1}^k R(k,m) \binom{2n+m}{2m+1}, \\
\Omega_{2k}(n) & = \sum_{m=1}^k R(k,m) \binom{n+m}{2m}, \\
S_{2k-1}(n) & = \sum_{m=1}^k R(k,m) \, \frac{1}{m} \binom{n+m}{2m}, \\[-2mm]
\intertext{and}
2^{2k-1} S_{2k-1}(n) & = \sum_{m=1}^k  Q_{k,m}(n) \binom{n+m}{2m-1},
\end{align*}
where
\begin{displaymath}
Q_{k,m}(n) = \begin{cases}
  n^{2k-1}, & \text{for $m=1$;} \\
  2 \, \displaystyle{\sum_{j=m}^k \binom{2k-1}{2j-2} R(j-1,m-1) n^{2k-2j+1}}, & \text{for $m \geq 2$.}
\end{cases}
\end{displaymath}

Furthermore, $S_{2k}(n)$ and $S_{2k+1}(n)$ can be expressed in the Faulhaber form $S_{2k}(n) = S_2(n) \sum_{r=1}^k b_{k,r} \big( S_1(n) \big)^{r-1}$ and $S_{2k+1}(n) = \big( S_1(n) \big)^2 \sum_{r=1}^k c_{k,r} \big( S_1(n) \big)^{r-1}$, with coefficients $b_{k,r}$ and $c_{k,r}$ given by
\begin{align*}
b_{k,r} & = \sum_{m=r}^k  \frac{3 \cdot 2^r}{(2m+1)!} \, R(k,m) Ps_{m}^{(r)}, \\[-3mm]
\intertext{and}
c_{k,r} & = \sum_{m=r}^k  \frac{2^{r+1}}{(2m+2)! (m+1)} \, R(k+1,m+1) Ps_{m+1}^{(r+1)},
\end{align*}
where $Ps_{m}^{(r)}$ are the Legendre-Stirling numbers of the first kind.

In addition, $\Omega_{2k}(n)$ can be expressed in the form $\Omega_{2k}(n) = \sum_{r=1}^k d_{k,r} \big( S_1(n) \big)^{r}$, with coefficients $d_{k,r}$ given by
\begin{equation*}
d_{k,r} = \sum_{m=r}^k  \frac{2^{r}}{(2m)!} \, R(k,m) Ps_{m}^{(r)}.
\end{equation*}

\section{Sums of powers of integers involving the sequence A304330}\label{sec:2}

In this section we successively prove the formulas for $2^{2k} S_{2k}(n)$, $S_{2k}(n)$, and $S_{2k-1}(n)$ stated in Theorem \ref{th:1} and Proposition \ref{pr:2} above. Then we obtain formulas for $T_{2k}(n)$ and $\Omega_{2k}(n)$, and derive a formula for $2^{2k-1}S_{2k-1}(n)$ as given in Theorem \ref{th:9}. Lastly, we present an alternative formula for $2^{2k-1}S_{2k-1}(n)$ in equation \eqref{odd2}.

\subsection{Proof of formula \eqref{th1}}

To prove \eqref{th1}, we need the following lemma.

\begin{lemma}\label{lm:3}
For integer $n \geq 1$, we have
\begin{equation*}
\sum_{i=1}^n (2i) \binom{2i+k-1}{2k-1} = k \binom{2n+k+1}{2k+1},
\end{equation*}
where $k$ is an arbitrary, fixed positive integer.
\end{lemma}

For later use we put the above binomial identity in the equivalent form
\begin{equation}\label{lm31}
\sum_{i=1}^n 4i (2i-k+1)(2i-k+2)\cdots (2i+k-1) = \frac{(2n-k+1)(2n-k+2) \cdots (2n+k+1)}{2k+1},
\end{equation}
and refer the reader to the Appendix for a proof of \eqref{lm31}.

For the proof of \eqref{th1}, we make use of two well-known properties of the (signed) central factorial numbers with even indices of the first kind $t(2k,2m)$ (OEIS A204579). In what follows we denote these numbers by $u(k,m)$, in accordance with the notation of Gelineau and Zeng \cite{gelineau}. The first property is the definition of $u(k,m)$ through the generating function
\begin{equation}\label{th11}
\sum_{m=1}^k u(k,m) x^m = x (x -1^2) (x -2^2) \cdots (x - (k-1)^2 ), \quad k \geq 1,
\end{equation}
(see, e.g., \cite[Equation (4.15)]{coffey} and Equation (11) (with $r=0$) of \cite{shiha}). The second property of interest is the central factorial inversion, which is a direct consequence of the orthogonality relations $\sum_{m=i}^k u(k,m) U(m,i) = \sum_{m=i}^k U(k,m) u(m,i) = \delta_{k,i}$ ($\delta_{k,i}$ being the Kronecker delta). This property is stated in the following lemma.

\begin{lemma}[{{\cite[Proposition 2.3]{butzer}}}]\label{lm:4}
If $(a_k )_{k \geq 1}$ and $(b_k )_{k \geq 1}$ are two sequences of real numbers (where we tacitly assume that $a_0 = b_0 =0$), there holds the inversion formula
\begin{equation*}
a_k = \sum_{m=1}^k u(k,m) b_m  \,\, \Longleftrightarrow \,\, b_k = \sum_{m=1}^k U(k,m) a_m, \quad k \geq 1.
\end{equation*}
\end{lemma}

Therefore, setting $x =(2i)^2$ in \eqref{th11} and multiplying through by $2$ gives
\begin{align*}
\sum_{m=1}^k u(k,m) 2^{2m+1} i^{2m} & = 8 i^2 ((2i)^2 -1^2) ((2i)^2 -2^2 ) \cdots
((2i)^2 - (k-1)^2 )   \\
& = 4i (2i-k+1) (2i-k+2) \cdots (2i+k-1).
\end{align*}
Furthermore, summing from $i = 1$ to $n$ on both sides of the last equation and employing \eqref{lm31} in the right-hand side, we get
\begin{equation*}
\sum_{m=1}^k u(k,m) 2^{2m+1} S_{2m}(n) = (2k)! \binom{2n+k+1}{2k+1}.
\end{equation*}
This allows us to apply Lemma \ref{lm:4} with $a_k = (2k)! \binom{2n+k+1}{2k+1}$ and $b_m = 2^{2m+1} S_{2m}(n)$ to finally obtain the following formula for $2^{2k} S_{2k}(n)$:
\begin{equation*}
2^{2k} S_{2k}(n) = \frac{1}{2} \sum_{m=1}^k U(k,m) (2m)! \binom{2n+m+1}{2m+1},
\end{equation*}
which, in view of \eqref{int2}, is just the formula \eqref{th1}.

As a result we find the following representation of $R(k,m)$, which follows straightforwardly by comparing the formulas \eqref{int1} and \eqref{th1}.
\begin{corollary}
For positive integers $k$ and $m$, with $1 \leq m \leq k$, we have
\begin{equation}\label{col5}
R(k,m) = 2^{2k+m-1} m! \sum \frac{(2k)!}{b_1! b_2! \cdots b_k!}
\prod_{r=1}^{k} \biggl( \frac{1}{4^r (2r)!} \biggl)^{b_r},
\end{equation}
where the summation extends over all $k$-tuples of nonnegative integers $(b_1,b_2,\ldots,b_k)$ fulfilling the conditions $b_1 + 2b_2 + \cdots + k b_k = k$ and $b_1 + b_2 +\cdots +b_k = m$.
\end{corollary}

Using the above representation, we get the following diagonals $R(k,k-s)$ for $s =0,1,2,3,4$:
\begin{align*}
R(k,k) & = \frac{1}{2} (2k)!,  \quad k \geq 1, \\
R(k,k-1) & = \frac{k-1}{24} (2k)!,  \quad k \geq 2, \\
R(k,k-2) & = \frac{(k-2)(5k-11)}{2880} (2k)!,  \quad k \geq 3, \\
R(k,k-3) & = \frac{(k-3)(35k^2 -231k +382)}{725760} (2k)!,  \quad k \geq 4, \\
R(k,k-4) & = \frac{(k-4)(175k^3 -2310k^2 + 10181k -14982)}{174182400} (2k)!, \quad k \geq 5.
\end{align*}

From the above particular cases, we can guess the general pattern
\begin{equation*}
R(k,k-s) = (2k)! (k-s) P_s(k), \,\,\, \text{for} \,\,\, k \geq s+1 \,\,\, \text{and} \,\,\, s\geq 1,
\end{equation*}
where $P_s(k)$ is a polynomial in $k$ of degree $s-1$ with coefficients having alternating signs.

\subsection{Proof of formulas \eqref{pr21} and \eqref{pr22}}

Next we prove \eqref{pr21}. For this, we need the following lemma.

\begin{lemma}\label{lm:6}
For integer $n \geq 1$, we have
\begin{equation*}
\sum_{i=1}^n i \binom{i+k-1}{2k-1} = \frac{k(2n+1)}{2k+1} \binom{n+k}{2k},
\end{equation*}
where $k$ is an arbitrary, fixed positive integer.
\end{lemma}

Likewise, for later use we put the above binomial identity in the equivalent form
\begin{equation}\label{lm51}
\sum_{i=1}^n 2i (i-k+1)(i-k+2)\cdots (i+k-1) = \frac{(2n+1)(n-k+1)(n-k+2) \cdots (n+k)}{2k+1},
\end{equation}
and refer the reader to the Appendix for a proof of \eqref{lm51}.

Now, setting $x =i^2$ in \eqref{th11} and summing from $i = 1$ to $n$ on both sides gives
\begin{equation*}
\sum_{m=1}^k u(k,m) S_{2m}(n) = \frac{1}{2} \sum_{i=1}^n 2i (i-k+1) (i-k+2) \cdots (i+k-1).
\end{equation*}
By virtue of \eqref{lm51}, this can be expressed as
\begin{equation*}
\sum_{m=1}^k u(k,m) S_{2m}(n) = \frac{(2k)!}{2} \, \frac{2n+1}{2k+1} \binom{n+k}{2k}.
\end{equation*}
Thus, taking $a_k = \frac{(2k)!}{2} \, \frac{2n+1}{2k+1} \binom{n+k}{2k}$ and $b_m = S_{2m}(n)$ in Lemma \ref{lm:4}, we obtain the following formula for $S_{2k}(n)$:
\begin{equation*}
S_{2k}(n) = \sum_{m=1}^k U(k,m) \frac{(2m)!}{2} \, \frac{2n+1}{2m+1} \binom{n+m}{2m},
\end{equation*}
which, by relation \eqref{int2}, is the same as formula \eqref{pr21}.

On the other hand, to show \eqref{pr22}, set $x =i^2$ in \eqref{th11} and divide both sides by $i$. This gives
\begin{equation*}
\sum_{m=1}^k u(k,m)  i^{2m-1} = (i-k+1) (i-k+2) \cdots (i+k-1).
\end{equation*}
Thus, summing over $i$ on both sides of the above equation and making use of the well-known binomial identity $\sum_{i=1}^n \binom{i+k-1}{2k-1} = \binom{n+k}{2k}$ (see \cite[Identity 58]{spivey}), we  get
\begin{equation*}
\sum_{m=1}^k u(k,m) S_{2m-1}(n) = (2k-1)! \binom{n+k}{2k}.
\end{equation*}
Finally, by inverting the equation above, we recover the known formula \cite[p.\ 284]{knuth}
\begin{equation*}
S_{2k-1}(n) = \sum_{m=1}^k (2m-1)! U(k,m) \binom{n+m}{2m},
\end{equation*}
which can be put in the form \eqref{pr22} upon using \eqref{int2}.

\begin{remark}
Formula \eqref{pr21} can be derived directly from \eqref{pr22} and the well-known recursive formula (see, e.g., \cite[Proposition 1]{ziel})
\begin{equation*}
S_k(n) = (n+1) S_{k-1}(n) - \sum_{i=1}^n S_{k-1}(i).
\end{equation*}
Indeed, putting $2k$ instead of $k$ in the above equation and substituting $S_{2k-1}(n)$ and $S_{2k-1}(i)$ from \eqref{pr22} gives
\begin{align*}
S_{2k}(n) & = (n+1)\sum_{m=1}^k  R(k,m) \, \frac{1}{m} \binom{n+m}{2m}
- \sum_{m=1}^k  R(k,m) \, \frac{1}{m} \sum_{i=1}^n \binom{i+m}{2m} \\
& = \sum_{m=1}^k  R(k,m) \, \frac{1}{m} \biggl( (n+1) \binom{n+m}{2m}
- \frac{n+m+1}{2m+1} \binom{n+m}{2m} \biggl) \\
& = \sum_{m=1}^k  R(k,m) \, \frac{2n+1}{2m+1} \binom{n+m}{2m}.
\end{align*}
\end{remark}

\subsection{Formulas for $T_{2k}(n)$ and $\Omega_{2k}(n)$}

For integer $k \geq 1$, the sum of the $k$th powers of the first $n$ odd integers $T_k(n) = 1^k + 3^k + \cdots + (2n-1)^k$ can be expressed as $T_k(n) = S_k(2n) - 2^k S_k(n)$. Consequently, for the case of even powers, from \eqref{th1} and \eqref{pr21} we have
\begin{equation*}
T_{2k}(n) = \sum_{m=1}^k R(k,m) \, \frac{4n+1}{2m+1} \binom{2n+m}{2m}
- \sum_{m=1}^k R(k,m) \binom{2n+m+1}{2m+1}.
\end{equation*}
Furthermore, since
\begin{equation*}
\frac{4n+1}{2m+1} \binom{2n+m}{2m} - \frac{2n+m+1}{2m+1} \binom{2n+m}{2m} = \frac{2n-m}{2m+1}
\binom{2n+m}{2m} = \binom{2n+m}{2m+1},
\end{equation*}
this becomes
\begin{equation}\label{oddi}
T_{2k}(n) = \sum_{m=1}^k R(k,m) \binom{2n+m}{2m+1}.
\end{equation}

\begin{remark}
Merca \cite[Equation (3.1)]{merca} proved that $T_{2k}(n) = \frac{2^{2k}}{2k+1} B_{2k+1}\big( n +\frac{1}{2} \big)$, where $B_k(n)$ is the $k$th Bernoulli polynomial evaluated at $n$. Thus, letting $n \to n - \frac{1}{2}$ in Merca's result, and using \eqref{oddi}, gives the identity
\begin{equation*}
B_{2k+1}(n) = \frac{2k+1}{2^{2k}} \sum_{m=1}^k R(k,m) \binom{2n+m-1}{2m+1}, \quad k \geq 1.
\end{equation*}
Note that this identity can also be obtained directly from \eqref{th1} and the well-known relationship $S_{2k}(n) = \frac{1}{2k+1} ( B_{2k+1}(n+1) - B_{2k+1})$, once one realizes that $B_{2k+1} =0$ for all $k \geq 1$.
\end{remark}

On the other hand, for integer $k \geq 1$, define the alternating sum of the $k$th powers of the first $n$ integers as
\begin{equation*}
\Omega_k(n) = n^k - (n-1)^k + \cdots + (-1)^{n-1} 1^k.
\end{equation*}
It is easily seen that $\Omega_k(2n) = S_k(2n) - 2T_k(n)$ and $\Omega_k(2n-1) = 2T_k(n) - S_k(2n-1)$. Hence, for the case of even powers, from \eqref{pr21} and \eqref{oddi} we find that $\Omega_{2k}(2n) = \sum_{m=1}^k R(k,m) \binom{2n+m}{2m}$ and $\Omega_{2k}(2n-1) = \sum_{m=1}^k R(k,m) \binom{2n+m-1}{2m}$. This means that
\begin{equation}\label{alter}
\Omega_{2k}(n) = \sum_{m=1}^k R(k,m) \binom{n+m}{2m},
\end{equation}
holds for all $k,n \geq 1$. In particular, $\Omega_2(n) = \binom{n+1}{2}$. See equation \eqref{vien1} below for an alternative representation of $\Omega_{2k}(n)$ as a polynomial in $S_1(n)$ of degree $k$.

\subsection{Corresponding formulas for $2^{2k-1}S_{2k-1}(n)$}

An explicit formula for $2^{2k-1} S_{2k-1}(n)$ involving the numbers $R(k,m)$ is given by the following theorem.
\begin{theorem}\label{th:9}
For integer $k \geq 1$, we have
\begin{equation}\label{th9}
2^{2k-1} S_{2k-1}(n) = \sum_{m=1}^k  Q_{k,m}(n) \binom{n+m}{2m-1},
\end{equation}
where $Q_{k,m}(n)$ is an odd polynomial in $n$ of degree $2k-2m+1$ given by
\begin{displaymath}
Q_{k,m}(n) = \begin{cases}
  n^{2k-1}, & \text{for $m=1$;} \\
  2 \, \displaystyle{\sum_{j=m}^k \binom{2k-1}{2j-2} R(j-1,m-1) n^{2k-2j+1}}, & \text{for $m \geq 2$.}
\end{cases}
\end{displaymath}
\end{theorem}
\begin{proof}
Formula \eqref{th9} follows readily by combining the representation of $R(k,m)$ in \eqref{col5} and the following formula for $2^{2k-1} S_{2k-1}(n)$ given in \cite[Theorem 3]{cere}, namely
\begin{multline*}
2^{2k-1} S_{2k-1}(n) = (n+1) n^{2k-1} + \sum_{j=1}^{k-1} 4^j n^{2k-2j-1} \binom{2k-1}{2j}
\sum \frac{(2j)!}{b_1! b_2! \cdots b_j!}  \\
\times \, \prod_{r=1}^{j} \biggl( \frac{1}{4^r (2r)!} \biggl)^{b_r} 2^m m! \binom{n+m+1}{2m+1},
\end{multline*}
where, for each $j =1,2,\ldots,k-1$, the rightmost summation is taken over all \mbox{$j$-tuples} of nonnegative integers $(b_1,b_2,\ldots,b_j)$ satisfying the constraint $b_1 + 2b_2 + \cdots + j b_j = j$, and where $m = b_1 + b_2 +\cdots +b_j$.
\end{proof}

It is to be noted that formula \eqref{th9} above actually involves a double summation, which makes it more complicated than its counterpart \eqref{th1}. We believe, however, that the structure of such a formula is interesting in its own right. In this respect, it is worth to emphasize the striking odd-parity property of the polynomials $Q_{k,m}(n)$. As an example, for $k=5$, we have
\begin{align*}
Q_{5,1}(n) & = n^9, \,\,\, Q_{5,2}(n) = 72n^7 + 252n^5 +168n^3 + 18n, \\[1mm]
Q_{5,3}(n) & = 3024 n^5 + 10080n^3 +4536n, \,\,\, Q_{5,4}(n)= 60480 n^3 + 90720n, \,\,\, Q_{5,5}(n)= 362880 n,
\end{align*}
from which it follows that
\begin{align*}
2^9 S_9(n) & = n^9 \binom{n+1}{1} + (72n^7 + 252n^5 +168n^3 + 18n) \binom{n+2}{3} \\
& + (3024 n^5 + 10080n^3 +4536n) \binom{n+3}{5} \\
& + (60480 n^3 + 90720n) \binom{n+4}{7} + 362880 n \binom{n+5}{9}.
\end{align*}

On the other hand, by setting $x =(2i)^2$ in \eqref{th11}, dividing through by $2i$, summing over $n$, and inverting the resulting equation, one arrives at the following formula for $2^{2k-1} S_{2k-1}(n)$:
\begin{equation}\label{odd2}
2^{2k-1} S_{2k-1}(n) = \sum_{m=1}^k  R(k,m) \, \frac{F_m(n)}{m},
\end{equation}
where
\begin{equation*}
F_m(n) = \sum_{i=1}^n \binom{2i+m-1}{2m-1},
\end{equation*}
is a polynomial in $n$ of degree $2m$. The first few of these polynomials are $F_1(n) = 2\binom{n+1}{2}$ and
\begin{align*}
F_2(n) & = \frac{1}{3} (2n^2 +2n -1) \binom{n+1}{2}, \\
F_3(n) & = \frac{2}{15} (8n^2 +8n -3) \binom{n+2}{4}, \\
F_4(n) & = \frac{1}{105} (8n^4 +16n^3 -24n^2 -32n +9) \binom{n+2}{4}, \\
F_5(n) & = \frac{2}{315} (16n^4 +32n^3 -46n^2 -62n +15) \binom{n+3}{6}, \\
F_6(n) & = \frac{1}{10395} (32n^6 +96n^5 -304n^4 -768n^3 +677n^2 +1077n -225) \binom{n+3}{6}, \\
F_7(n) & = \frac{2}{135135} (256n^6 +768n^5 -2400n^4 -6080n^3 +5168n^2 +8336n -1575) \binom{n+4}{8},
\end{align*}
from which it can be inferred that $F_m(n)$ factorizes as
\begin{equation*}
F_m(n) = G_m(n) \binom{n + \lfloor \frac{m+1}{2} \rfloor}{2 \lfloor \frac{m+1}{2} \rfloor}, \quad m \geq 1,
\end{equation*}
where $\lfloor \cdot \rfloor$ denotes the floor function and $G_m(n)$ is a polynomial in $n$ of degree $2m - 2 \lfloor \frac{m+1}{2} \rfloor$. Unfortunately, there appears to be no recognizable closed-form expression for $G_m(n)$.

For comparison, we write down the expression of $2^9 S_9(n)$ that is obtained from \eqref{odd2}
\begin{align*}
2^9 S_9(n) & = (340n^2 +340n -168) \binom{n+1}{2} \\
& + (11520 n^4 + 23040n^3 -15744n^2 -27264n +5904) \binom{n+2}{4} \\
& + (36864 n^4 + 73728n^3 -105984n^2 -142848n +34560)\binom{n+3}{6}.
\end{align*}

\section{Faulhaber form of $S_{2k}(n)$ and $S_{2k+1}(n)$}\label{sec:3}

Edwards' seminal work \cite{edw} triggered an intensive research on the so-called Faulhaber polynomials, named after the German mathematician Johann Faulhaber (1580--1635); see, e.g., \cite{bazso}, \cite[Section 3]{beardon}, \cite{chen,coffey}, \cite[Section 12]{gessel}, and \cite{kellner,knuth,krishna,ziel}). As is well-known, the power sums $S_{2k}(n)$ and $S_{2k+1}(n)$ (with $k \geq 1$) can be expressed in the Faulhaber form as
\begin{align}
S_{2k}(n) & = S_2(n)  \big[ b_{k,1} + b_{k,2} S_1(n) + b_{k,3} \big( S_1(n) \big)^2 + \cdots + b_{k,k} \big( S_1(n)
\big)^{k-1} \big],  \label{fau1}  \\[-2mm]
\intertext{and}
S_{2k+1}(n) & =  \big ( S_1(n) \big)^2  \big[ c_{k,1} + c_{k,2} S_1(n) + c_{k,3} \big( S_1(n) \big)^2 + \cdots + c_{k,k}
\big( S_1(n) \big)^{k-1} \big], \label{fau2}
\end{align}
where the Faulhaber coefficients $b_{k,r}$ and $c_{k,r}$, $r =1,2,\ldots,k$, are (nonzero) rational numbers satisfying the relation
\begin{equation}\label{fau3}
b_{k,r} = \frac{3r+3}{4k+2} \, c_{k,r}.
\end{equation}

\begin{remark}
An explicit formula for $c_{k,r}$ in terms of the Bernoulli numbers was first obtained by Gessel and Viennot \cite[Equation (12.10)]{gessel}. In our notation, this last formula can be rewritten in the form (cf.\ \cite[Equation (3.5)]{cere2} and \cite[Equation (5-3)]{kellner})
\begin{equation*}
c_{k,r} = (-1)^{r-1} \frac{2^{r+1}}{r+1} \sum_{m=0}^{\lfloor \frac{r-1}{2} \rfloor} \binom{2r-1-2m}{r}
\binom{2k+1}{2m+1} B_{2k-2m}.
\end{equation*}
\end{remark}

In the next proposition we show how Knuth's formulas \eqref{pr21} and \eqref{pr22} for, respectively, $S_{2k}(n)$ and $S_{2k+1}(n)$, can be converted into the corresponding Faulhaber form by exploiting the relationship between $\binom{n+m}{2m}$ and the Legendre-Stirling numbers of the first kind.
\begin{proposition}\label{prop:11}
For integer $k \geq 1$, the power sum polynomials $S_{2k}(n)$ and $S_{2k+1}(n)$ can be expressed in the Faulhaber form \eqref{fau1} and \eqref{fau2}, respectively, with coefficients $b_{k,r}$ and $c_{k,r}$ given by
\begin{align}
b_{k,r} & = \sum_{m=r}^k  \frac{3 \cdot 2^r}{(2m+1)!} \, R(k,m) Ps_{m}^{(r)}, \label{coeb} \\
c_{k,r} & = \sum_{m=r}^k  \frac{2^{r+1}}{(2m+2)! (m+1)} \, R(k+1,m+1) Ps_{m+1}^{(r+1)}, \label{coec}
\end{align}
for $r =1,2, \ldots,k$, and where $Ps_{m}^{(r)}$ are the Legendre-Stirling numbers of the first kind.
\end{proposition}
\begin{proof}
Let us first recall that, according to \cite[Theorem 5.5]{andrews}, the (signed) Legendre-Stirling numbers of the first kind $Ps_{m}^{(r)}$ (OEIS A129467) satisfy the horizontal generating function (for integers $r$ and $m$ fulfilling $1\leq r \leq m$)
\begin{equation*}
\langle x \rangle_m = \sum_{r=1}^m Ps_m^{(r)} x^r,
\end{equation*}
where $\langle x \rangle_m$ is the generalized falling factorial defined by
\begin{equation*}
\langle x \rangle_m = \prod_{j=0}^{m-1} (x - j(j+1)).
\end{equation*}
(For the sake of illustration, the first values of $Ps_{m}^{(r)}$ are listed in Table \ref{tab:2}.) Now, invoking the binomial expansion (see \cite[Equation (5.1)]{coffey})
\begin{equation}\label{fau4}
\binom{n+m}{2m} = \frac{1}{(2m)!} \sum_{r=1}^m Ps_m^{(r)} \big( n(n+1) \big)^r,
\end{equation}
and combining \eqref{fau4} with \eqref{pr21}, we obtain
\begin{align*}
S_{2k}(n) & = (2n+1) \sum_{m=1}^k \sum_{r=1}^m \frac{2^r}{(2m+1)!} \, R(k,m) Ps_{m}^{(r)} \big( S_1(n) \big)^r  \\
& = S_2(n) \sum_{r=1}^k  \Biggl( \sum_{m=r}^k  \frac{3 \cdot 2^r}{(2m+1)!} \, R(k,m) Ps_{m}^{(r)} \Biggl) \big( S_1(n) \big)^{r-1},
\end{align*}
which is of the form \eqref{fau1} with the coefficients $b_{k,r}$ given by \eqref{coeb}.

\begin{table}[ttt]
\centering
\begin{tabular}{|c||c|c|c|c|c|c|c|}
\hline
$m \backslash r$ & $r=0$ & $r=1$ & $r=2$ & $r=3$ & $r=4$ & $r=5$ & $r=6$ \\
\hline\hline
$m=0$ & 1 & 0 & 0 & 0 & 0 & 0 & 0  \\ \hline
$m=1$ & 0 & 1 & 0 & 0 & 0 & 0 & 0  \\ \hline
$m=2$ & 0 & $-2$ & 1 & 0 & 0 & 0 & 0  \\ \hline
$m=3$ & 0 & 12 & $-8$ & 1 & 0 & 0 & 0  \\ \hline
$m=4$ & 0 & $-144$ & 108 & $-20$ & 1 & 0 & 0  \\ \hline
$m=5$ & 0 & 2880 & $-2304$ & 508 & $-40$ & 1 & 0  \\ \hline
$m=6$ & 0 & $-86400$ & 72000 & $-17544$ & 1708 & $-70$ & 1 \\ \hline
\end{tabular}
\caption{The Legendre-Stirling numbers of the first kind $Ps_{m}^{(r)}$ up to $m=6$.}\label{tab:2}
\end{table}

Likewise, by using \eqref{fau4} in \eqref{pr22}, we obtain
\begin{align*}
S_{2k+1}(n) & = \sum_{m=1}^{k+1} \sum_{r=1}^m \frac{2^r}{(2m)! m} \, R(k+1,m) Ps_{m}^{(r)} \big( S_1(n) \big)^r  \\
& = \sum_{r=1}^{k+1} \Biggl( \sum_{m=r}^{k+1}  \frac{2^r}{(2m)! m} \, R(k+1,m) Ps_{m}^{(r)} \Biggl ) \big( S_1(n) \big)^{r}.
\end{align*}
The key point to observe here is that the term in $S_1(n)$ in the above expansion of $S_{2k+1}(n)$ does vanish because the derivative of $S_{2k+1}(n)$ with respect to $n$ (considering $n$ as a continuous variable) evaluated at $n=0$, namely, $S_{2k+1}^{\prime}(0) = (-1)^{2k+1} B_{2k+1}$, is equal to zero for all $k \geq 1$. This means that
\begin{equation}\label{fau5}
\sum_{m=1}^{k+1}  \frac{1}{(2m)! m} \, R(k+1,m) Ps_{m}^{(1)} =0,
\end{equation}
and then
\begin{align*}
S_{2k+1}(n)  & = \sum_{r=2}^{k+1} \Biggl( \sum_{m=r}^{k+1} \frac{2^r}{(2m)! m} \, R(k+1,m)
Ps_{m}^{(r)} \Biggl)
\big( S_1(n) \big)^{r} \\
& = \big( S_1(n) \big)^2 \sum_{r=1}^{k} \Biggl( \sum_{m=r}^{k}  \frac{2^{r+1}}{(2m+2)! (m+1)} \,
R(k+1,m+1) Ps_{m+1}^{(r+1)} \Biggl) \big( S_1(n) \big)^{r-1},
\end{align*}
which is of the form \eqref{fau2} with the coefficients $c_{k,r}$ given by \eqref{coec}.
\end{proof}

\begin{remark}
Thanks to relation \eqref{fau3}, the coefficients $c_{k,r}$ can also be expressed as
\begin{equation*}
c_{k,r} = \frac{2k+1}{r+1} \sum_{m=r}^k  \frac{2^{r+1}}{(2m+1)!} \, R(k,m) Ps_{m}^{(r)}.
\end{equation*}
\end{remark}

On the other hand, substituting the expression in \eqref{fau4} for $\binom{n+m}{2m}$ into \eqref{alter}, we obtain the following Faulhaber-like formula for the alternating sum of even powers $\Omega_{2k}(n)$:
\begin{equation}\label{vien1}
\Omega_{2k}(n) = \sum_{r=1}^k d_{k,r} \big( S_1(n) \big)^{r}, \quad k \geq 1,
\end{equation}
where the coefficients $d_{k,r}$ are given by
\begin{equation*}
d_{k,r} = \sum_{m=r}^k  \frac{2^{r}}{(2m)!} \, R(k,m) Ps_{m}^{(r)}.
\end{equation*}

\begin{remark}
Gessel and Viennot \cite[Theorem 28]{gessel} established the following formula for $\Omega_{2k}(n)$:
\begin{equation}\label{vien2}
\Omega_{2k}(n) = \frac{1}{2} \sum_{r=1}^k \mathfrak{s}(k,r) \big( n(n+1)\big)^r, \quad k \geq 1,
\end{equation}
where $\mathfrak{s}(n,r)$ denotes the so-called Sali\'{e} numbers (OEIS A065547). By comparing \eqref{vien1} and \eqref{vien2}, and recalling \eqref{int2}, we get the following representation of the (signed) Sali\'{e} numbers in terms of $U(k,m)$ and $Ps_m^{(r)}$:
\begin{equation*}
\mathfrak{s}(k,r) = \sum_{m=r}^k U(k,m) Ps_m^{(r)}.
\end{equation*}
\end{remark}

\section{Concluding remarks}\label{sec:4}

We conclude this paper with the following remarks:

From \eqref{fau5} one can deduce the following horizontal recurrence relation for the numbers $R(k,m)$.
\begin{proposition}
For integer $k \geq 2$, we have
\begin{equation*}
\sum_{m=1}^k (-1)^m \frac{( (m-1)! )^2}{(2m)!} \, R(k,m) = 0.
\end{equation*}
\end{proposition}
\begin{proof}
This follows immediately by setting $k \to k-1$ in \eqref{fau5} and using the fact that $ Ps_{m}^{(1)} = (-1)^{m-1} m!(m-1)!$ \cite[Theorem 2]{guo}.
\end{proof}

Let us further note that the above recurrence relation can be written in terms of $U(k,m)$ as
\begin{equation*}
\sum_{m=1}^k (-1)^m ( (m-1)! )^2  U(k,m) = 0,  \quad k \geq 2.
\end{equation*}

Moreover, since $b_{k,1} = 6 B_{2k}$, it follows from \eqref{coeb} that
\begin{equation}\label{b2k}
B_{2k} = \sum_{m=1}^k (-1)^{m-1} \frac{m!(m-1)!}{(2m+1)!} \, R(k,m),  \quad k \geq 1,
\end{equation}
in accordance with the formula for $B_{2k}$ given by Coffey et al.\ \cite[p.\ 29]{coffey}. Hence, combining \eqref{b2k} with the classical Bernoulli's formula \cite[p.\ 48]{knoebel}, $S_k(n) = \frac{1}{k+1} n^{k+1} + \frac{1}{2}n^k + \frac{1}{k+1} \sum_{j=2}^{k} \binom{k+1}{j} B_j n^{k+1-j}$, we obtain that
\begin{equation*}
S_k(n) = \frac{n^{k+1}}{k+1} + \frac{1}{2}n^k - \frac{1}{k+1} \sum_{j=1}^{\lfloor \frac{k}{2} \rfloor}
\sum_{m=1}^j (-1)^m \frac{m!(m-1)!}{(2m+1)!} \binom{k+1}{2j} R(j,m) n^{k+1-2j},
\end{equation*}
which holds for all $k \geq 1$. Note that setting $n=1$ and $k = 2r+1$ in the above formula gives the identity
\begin{equation*}
\sum_{j=1}^r \sum_{m=1}^j (-1)^{m-1} \frac{m!(m-1)!}{(2m+1)!}
\binom{2r+2}{2j} R(j,m) = r, \quad r \geq 1.
\end{equation*}

\section{Acknowledgments}

The author would like to thank the anonymous referee and the editor-in-chief for the careful reading of a first version of the paper and for many insightful comments and suggestions that greatly improved the quality of the paper. The author also wants to thank Dionisio F.\ Y\'{a}\~{n}ez (University of Valencia) for sending me a copy of the paper \cite{butzer}.

\appendix\section{Appendix}

In this section we prove the relations \eqref{lm31} and \eqref{lm51}, which are equivalent to the binomial identities appearing in Lemma \ref{lm:3} and Lemma \ref{lm:6}, respectively.

\subsection{Proof of relation \eqref{lm31}}

We prove \eqref{lm31} by induction on $n$. To this end, we distinguish the cases of even and odd $k$.

\begin{itemize}[leftmargin=*]

\item Even $k$. It is evident that both sides of \eqref{lm31} are identically zero when $n < \frac{k}{2}$. Therefore, it is enough to show that
\begin{equation}\label{lm32}
\sum_{i= \frac{k}{2}}^n 4i (2i-k+1)(2i-k+2)\cdots (2i+k-1) = \frac{(2n-k+1)(2n-k+2) \cdots (2n+k+1)}{2k+1},
\end{equation}
for $n \geq \frac{k}{2}$. Relation \eqref{lm32} holds trivially for $n = \frac{k}{2}$. Assume the statement is true for $n = \frac{k}{2} + m$, where $m$ is an arbitrary, fixed nonnegative integer. We prove it for $n = \frac{k}{2} + m +1$
\begin{align*}
\sum_{i= \frac{k}{2}}^{\frac{k}{2} +m+1} & 4i (2i-k+1)(2i-k+2)\cdots (2i+k-1) \\[-2mm]
 & = \sum_{i= \frac{k}{2}}^{\frac{k}{2} +m} 4i (2i-k+1)(2i-k+2)\cdots (2i+k-1) \\
 & \qquad + (2k +4m+4)(2m+3)(2m+4) \cdots (2m+2k+1) \\[2mm]
 & = \frac{(2m+1)(2m+2) \cdots (2m+2k+1)}{2k+1} \\
 & \qquad + (2k +4m+4)(2m+3)(2m+4) \cdots (2m+2k+1) \\[2mm]
 & = \frac{(2m+3)(2m+4) \cdots (2m+2k+3)}{2k+1}.
\end{align*}
Thus \eqref{lm32} holds for all $n \geq \frac{k}{2}$ by induction.

\item Odd $k$. Likewise, for odd $k$ it is enough to show that
\begin{equation}\label{lm33}
\sum_{i=\frac{k+1}{2}}^n 4i (2i-k+1)(2i-k+2)\cdots (2i+k-1) = \frac{(2n-k+1)(2n-k+2) \cdots (2n+k+1)}{2k+1},
\end{equation}
for $n \geq \frac{k+1}{2}$. Clearly, relation \eqref{lm33} is satisfied for $n = \frac{k+1}{2}$. Assuming that \eqref{lm33} is true for $n = \frac{k+1}{2}+m$ (for any given nonnegative integer $m$), we prove it for $n = \frac{k+1}{2}+m +1$
\begin{align*}
\sum_{i=\frac{k+1}{2}}^{\frac{k+1}{2}+m+1} & 4i (2i-k+1)(2i-k+2)\cdots (2i+k-1) \\[-2mm]
 & = \sum_{i=\frac{k+1}{2}}^{\frac{k+1}{2}+m} 4i (2i-k+1)(2i-k+2)\cdots (2i+k-1) \\
 & \qquad + (2k +4m+6)(2m+4)(2m+5) \cdots (2m+2k+2) \\[2mm]
 & = \frac{(2m+2)(2m+3) \cdots (2m+2k+2)}{2k+1} \\
 & \qquad + (2k +4m+6)(2m+4)(2m+5) \cdots (2m+2k+2) \\[2mm]
 & = \frac{(2m+4)(2m+5) \cdots (2m+2k+4)}{2k+1},
\end{align*}
which means that \eqref{lm33} holds for all $n\geq \frac{k+1}{2}$ by induction. This completes the proof of \eqref{lm31}.
\end{itemize}

\subsection{Proof of relation \eqref{lm51}}

Similarly, we prove \eqref{lm51} by induction on $n$. Note that both sides of \eqref{lm51} are identically zero whenever $n < k$. Therefore, it suffices to show that
\begin{equation}\label{lm52}
\sum_{i=k}^n 2i (i-k+1)(i-k+2)\cdots (i+k-1) = \frac{(2n+1)(n-k+1)(n-k+2) \cdots (n+k)}{2k+1},
\end{equation}
for $n \geq k$. Clearly, \eqref{lm52} is correct for $n =k$. Assume \eqref{lm52} holds for $n = k +m$, with $m$ being any given nonnegative integer. Then
\begin{align*}
\sum_{i=k}^{k +m+1} & 2i (i-k+1)(i-k+2)\cdots (i+k-1) \\[-2mm]
 & = \sum_{i=k}^{k +m} 2i (i-k+1)(i-k+2)\cdots (i+k-1) \\
 & \qquad + (2k+2m+2)(m+2)(m+3) \cdots (m+2k) \\[2mm]
 & = \frac{(2k+2m+1)(m+1)(m+2) \cdots (m+2k)}{2k+1} \\
 & \qquad + (2k +2m+2)(m+2)(m+3) \cdots (m+2k) \\[2mm]
 & = \frac{(2k+2m+3)(m+2)(m+3) \cdots (m+2k+1)}{2k+1}.
\end{align*}
Thus \eqref{lm52} holds for all $n \geq k$ by induction, and \eqref{lm51} is proven.

\end{document}